\documentclass[11pt,leqno]{article}

\usepackage[all]{xy}
\usepackage{amssymb,amsmath,amsthm,    url,rotating}
\usepackage{mathrsfs}
  \usepackage{graphicx,epsfig}
  \usepackage{xcolor,graphicx}

    \usepackage{makeidx}
\newcommand{\vp}{\varepsilon}

\newcommand{\cl}[1]{\mathcal{#1}}

\theoremstyle{plain}
\newtheorem{thm}{Theorem}[section]
\newtheorem{lem}[thm]{Lemma}

\theoremstyle{definition}

\theoremstyle{remark}

\numberwithin{equation}{section}

\def\C{\bb  C}

\def\phi{\varphi}

\begin{document}

\title{A note on $C^*$-algebras with Lifting Property}

\author{by\\
Gilles  Pisier\footnote{ORCID    0000-0002-3091-2049}   \\
Sorbonne Universit\'e\\
and\\ Texas  A\&M  University}

\def\C{\mathscr{C}}
\def\B{\mathscr{B}}
\def\I{\cl  I}
\def\e{\cl  E}
 
  \def\a{\alpha}
 
       \def\t{\theta}
       
  \maketitle
  
\begin{abstract}      
We give several simple and easy complements 
to our recent paper on $C^*$-algebras with the lifting property (LP in short).
In particular we observe that the local lifting property (LLP in short)
associated to the class of max-contractions implies the
lifting property for max-contractions.
\end{abstract}
  
  \medskip{MSC (2010): 46L06, 46L07, 46L09} 
  
    \medskip

  \def\L{\cl L}
  
  
  Suppose we are given a class of maps $\cl F$
  that includes maps defined from an operator subspace
  $E\subset A$ of a $C^*$-algebra into another $C^*$-algebra $B$.
  The latter maps are denoted by  $\cl F(E,B)$ so that we may think of
  $\cl F$ as the union of all possible $\cl F(E,B)$ (but we need to remember
  the ambient $A$ for $E$).
  
  Let $C$ be another $C^*$-algebra.
  We denote by  $C \otimes_{\max} E$
  (resp.  $E \otimes_{\max} C$) the closure of
  the algebraic tensor product $C \otimes E$ (resp.  $E \otimes  C$) in $C \otimes_{\max} A$
  (resp. $A  \otimes_{\max} C$). 
  We should warn the reader that   this definition depends on the embedding
  $E\subset A$ and not only on the operator space structure of $E$.
  
A unital $C^*$-algebra
  $A$ is said to have the  $\cl F$-LP (resp.  $\cl F$-LLP)
  if for any unital quotient $C^*$-algebra $C/I$ and any map
  $u\in \cl F(A,C/I)$ there is a map $\hat u\in \cl F(A,C)$ that lifts $u$
  (resp. for any finite dimensional operator system
  $E\subset A$ the restriction $u_{|E}$ admits a lifting in $\cl F(E,C)$).
  
  This definition is modeled on the ones introduced by Kirchberg \cite{Kir}
  for the class $\cl F_{ucp}$ of unital c.p. maps.
  In that case we refer to  $\cl F$-LP (resp.  $\cl F$-LLP) simply as LP (resp. LLP).
  
  We are particularly interested in the question whether 
  the existence of local liftings implies that of liftings for various $\cl F$.
 Whether this is true  turns out to be often a well known open problem. 
  It is open for instance when $\cl F$ is either the class of all contractions
  (i.e. $\cl F=\{ u: E \to B \mid \|u\|\le 1\}$ 
  or that of positive contractions.
  
  Whether $\cl F$-LLP implies $\cl F$-LP is also open
  when $\cl F$ is the class of all complete contractions (i.e. maps such that
  $\|u\|_{cb} \le 1$). In that case the $\cl F$-LP and the $\cl F$-LLP
  have been extensively  studied by Ozawa in \cite{Oz} 
  under the names of OLP and OLLP respectively.
  
  See the references for more recent information on the LP and the LLP.
  
  The question whether LLP implies LP seems connected to the Connes embedding problem.
   The recent paper \cite{JNVWY}     
  posted on arxiv in Jan. 2020 by
 Ji, Natarajan, Vidick,  Wright, and  Yuen   contains a negative solution to the 
 Connes embedding problem. See also \cite{Vid}.\\
 Indeed, as pointed out in \cite{Kir} the implication [LLP implies LP] would hold if the Connes embedding problem had an affirmative answer.
  This is the main motivation for our trying to disprove the latter implication.
  A direct proof would have the great advantage of producing a hopefully simpler negative solution to the  Connes embedding problem,
  but another possible route could be to prove that actually the validity of
   [LLP implies LP] implies a positive solution to the Connes embedding problem. One would then solve the  [LLP implies LP] problem negatively
   as a consequence of the results of \cite{JNVWY,Vid}.
   There remains the possibility that (as conjectured by Kirchberg)  [LLP implies LP] is true. Of course
   this would be a major result.
  
  In \cite{157} we introduced the class of maximally bounded (m.b. in short) maps in parallel with that of completely bounded (c.b. in short) maps.
  A map $u: E \to B$ is said to be m.b. if there is a constant $c$ such that
  for any $t\in \C \otimes E$ (algebraic tensor product)
  we have
  $$\| [Id \otimes u](t) \|_{\C \otimes_{\max}{} B}\le c \|t\|_{\C \otimes_{\max}{} E}.$$
When this holds it actually holds automatically for any $C^*$-algebra $C$
in place of $\C$.\\
We denote by $\|u\|_{mb}$ the smallest $c$ for which this holds, and 
by $MB(E,B)$ the space of m.b. maps equipped with this norm.

  \begin{thm} A separable unital $C^*$-algebra
  $A$ has the LP if and only if 
  for any f.d. subspace $E\subset A$ and any   $C^*$-algebra
  $C$ we have  a contractive inclusion
\begin{equation}\label{1} MB(E, C^{**} )\subset  MB(E,C) ^{**}.\end{equation}
  Moreover,  it suffices for this to hold to have it for $C=\C$.
  \end{thm}
  
  For a proof see \cite[Th. 3.1]{157} and   \cite[Cor. 4.6]{157}.
  As observed in \cite[Rem. 3.3]{157} the converse inclusion holds in full generality.

  Let us denote by $\cl F_{mc}$ the class of all maps $u$ with $\|u\|_{mb}
  \le 1$, that we call $\max$-contractions.
  
  \begin{thm} Let  $A$  be a separable unital $C^*$-algebra. The following are equivalent:
 \begin{itemize}
 \item [{\rm(i) } ] $A$ has the LP.
 \item[{\rm (ii) }] $A$ has the $\cl F_{mc}$-LP.
 \item[ {\rm(iii) }] $A$ has the $\cl F_{mc}$-LLP.
 \end{itemize}
  \end{thm}
  \begin{proof}
  Assume (i). Let $u\in MB(A,C/\cl I)$ with $\|u\|_{mb} \le 1$.
  Then by Kirchberg's characterization  of mb-maps (see \cite[Th. 14.1, p. 261]{P4}), 
  $u$ has dec-norm $\le 1$ when considered as taking values in
  $(C/\cl I)^{**}$.
  We have $ C^{**}\simeq  {\cl I}^{**}\oplus  (C/\cl I)^{**}$.
  Therefore there is a map $v: A \to C^{**}$ that lifts the map
  $ju: A \to (C/\cl I)^{**}$ with $j: C/\cl I \to (C/\cl I)^{**}$ denoting the inclusion.
  Consider now $v_{|E} : E \to C^{**}$. We now repeat a standard argument:
  by \eqref{1} we have a net of maps $v_i: E \to C$ with $\|v_i\|_{mb}\le1$
  tending weak* to $v_{|E} $, and hence, denoting by $q$ the quotient map,
  such that $qv_i$ tends weakly to $u_{|E}$. By Mazur's theorem, passing to convex combinations,
  we can  modify the maps $v_i: E \to C$ so that $qv_i$ tends (pointwise) in norm to $u_{|E}$. But now if $A$ has the LP
  by \cite[Th. 3.5]{157} for any $\vp>0$ the maps $v_i: E \to C$ extend to $w_i: A \to C$
  with  $\|w_i\|_{mb} \le 1+\vp$.
  Now a moment of thought (renormalizing and letting $\vp\to 0$) shows that 
  our initial map $u$ is such that $qu$ 
  can be approached pointwise   in norm  by maps of the form $qw$
  with $w$ in the unit ball of $MB(A,C)$.
Then Arveson's well known principle, ``pointwise limits of liftable maps are liftable" can be applied in this instance to show that
there is $\hat u$  in the unit ball of $MB(A,C)$ lifting $u$. This shows (i) implies (ii) (actually we showed that (i) implies (iii) and that (i) and (iii) together imply (ii)).
\\
(ii) implies (iii) is obvious.
  \\Assume (iii). Consider $u\in MB(E, C^{**} )$
  with $E\subset A$.
  Then by the extension property of mb maps (see \cite[Th. 14.6, p. 266]{P4} or \cite[p. 125 ]{154})
  $u$ extends to a dec map  $v: A \to (C^{**})^{**}$
  with   $\|v\|_{dec}\le 1$, meaning if $i_2: C^{**} \to (C^{**})^{**}$ denotes  the canonical inclusion we have $ v_{|E}=i_2 u$. Clearly there is 
  a u.c.p. projection $P:  (C^{**})^{**} \to C^{**}$, so that $Pv_{|E}= u$.
  Thus the map $u'=Pv$ extends $u$ and
    $\|u'\|_{mb} \le \|P\|_{mb}\|v\|_{mb}\le 1$. By Lemma \ref{klem} below
   we have $C^{**}=V/\cl I$ with $\cl I=\ker(Q)$. By the local liftability
    of $u': A \to V/\cl I$ there is a map ${u'}^E: E \to V$ 
    with mb-norm $\le 1$ lifting $u'_{|E}$. Composing ${u'}^E$ with the inclusion $V\subset \ell_\infty(I, C)$, we find a net
    of maps $v_i: E \to C$ with $\|v_i\|_{mb} \le 1$ tending weak*
    to $u'_{|E}$. But $u'_{|E}=u_{|E}$ since $u'$ extends $u$.
This proves that $A$ satisfies \eqref{1}, and hence has the LP.
  \end{proof}
  
  
  We will now include for the reader's convenience the 
  proof   
    (see e.g. \cite[p. 205 or p. 445]{157} for more details) 
    of the following probably well known fact, used above.

  \begin{lem}\label{klem} Let $C$ be any unital $C^*$-algebra. For some suitable set $I$,
  there is a unital $C^*$-subalgebra $V\subset \ell_\infty(I, C)$ and
  a surjective unital $*$-homomorphism $Q: V \to C^{**}$.
  Moreover $Q$ extends to a u.c.p. map on $\ell_\infty(I, C)$.

  \end{lem}
   \begin{proof}
 Let  $I$ be the set of all neighbourhoods
     of $0$ in the weak* topology of $C^{**}$. 
    Let $\cl U$ be an ultrafilter on $I$ finer than
     the latter filter of neighbourhoods.
    Let $V $ be the set of all $x=(x_i) \subset  \ell_\infty(I, C)$ such that
    $\lim_\cl U x_i $ exists in the weak* topology of $C^{**}$.
     Then the map
    $Q: \ell_\infty(I, C) \ni (x_i) \mapsto \lim_\cl U x_i $ is a u.c.p. map from
    $\ell_\infty(I, C)$ 
      to $C^{**}$. By the well known weak* density of the unit ball of a space in its bidual,   $Q$ maps the closed unit ball onto the closed unit ball. Let $V$ be its multiplicative domain, so that $Q_{|V}$ is a $*$-homomorphism.
   Letting $\cl I=\ker(Q_{|V})$ we obtain the result
   because all unitaries of  $C^{**}$ are in the range of
   $Q_{|V}$
   (see \cite[Th. 7.33 p.149]{157}).
     \end{proof}
     
     It might be useful to record that the subspace $V\subset \ell_\infty(I, C)$ is max-injective, meaning
     that for any other $C^*$-algebra $D$ (in particular $\C$)
     we have $D\otimes_{\max} V \subset D \otimes_{\max} \ell_\infty(I, C)$
     isometrically.

  \end{document}